\documentclass[11pt]{article}
\usepackage{amsmath, amssymb, amsthm}
\usepackage{verbatim}
\usepackage{multicol}
\usepackage{enumerate}
\usepackage{comment}
\usepackage[none]{hyphenat}
\usepackage{hyperref}
\hypersetup{
	colorlinks=true,
	linkcolor=blue,
	filecolor=magenta,
	urlcolor=cyan,
	citecolor=blue
}

\usepackage{pgf}
\usepackage{tikz}
\usepackage{ifthen}
\usetikzlibrary{math}
\usetikzlibrary{positioning,arrows,shapes,decorations.markings,decorations.pathreplacing,matrix,patterns}
\tikzstyle{vertex}=[circle,draw=black,fill=black,inner sep=0,minimum size=3pt,text=white,font=\footnotesize]

\date{}
\title{\vspace{-0.8cm}Ordered graphs and large bi-cliques in intersection graphs of curves}
\author{
	J\'{a}nos Pach \thanks{\'{E}cole Polytechnique F\'{e}d\'{e}rale de Lausanne, Research partially supported by Swiss National Science Foundation grants no. 200020-162884 and 200021-175977. \emph{e-mail}: \textbf{\{janos.pach, istvan.tomon\}@epfl.ch}} \thanks{R\'enyi Institute of Hungarian Academy of Sciences}
	\and
	Istv\'{a}n Tomon \footnotemark[1]	
}

\theoremstyle{plain}
\newtheorem{theorem}{Theorem}
\newtheorem{definition}[theorem]{Definition}
\newtheorem{corollary}[theorem]{Corollary}
\newtheorem{claim}[theorem]{Claim}
\newtheorem{lemma}[theorem]{Lemma}
\newtheorem{conjecture}[theorem]{Conjecture}

\theoremstyle{definition}

\begin{document}

\maketitle
\sloppy

\begin{abstract}
An {\em ordered graph} $G_<$ is a graph with a total ordering $<$ on its vertex set. A {\em monotone path} of length $k$ is a sequence of vertices $v_1<v_2<\ldots<v_k$ such that $v_iv_{j}$ is an edge of $G_<$ if and only if $|j-i|=1$. A {\em bi-clique} of size $m$ is a complete bipartite graph whose vertex classes are of size $m$.

We prove that for every positive integer $k$, there exists a constant $c_k>0$ such that every ordered graph on $n$ vertices that does not contain a monotone path of length $k$ as an induced subgraph has a vertex of degree at least $c_kn$, or its complement has a bi-clique of size at least $c_kn/\log n$. A similar result holds for ordered graphs containing no induced ordered subgraph isomorphic to a fixed ordered matching.

As a consequence, we give a short combinatorial proof of the following theorem of Fox and Pach. There exists a constant $c>0$ such the intersection graph $G$ of any collection of $n$ $x$-monotone curves in the plane has a bi-clique of size at least $cn/\log n$ or its complement contains a bi-clique of size at least $cn$. (A curve is called $x$-monotone if every vertical line intersects it in at most one point.) We also prove that if $G$ has at most $\left(\frac14 -\epsilon\right){n\choose 2}$ edges for some $\epsilon>0$, then     $\overline{G}$ contains a linear sized bi-clique. We show that this statement does not remain true if we replace $\frac14$ by any larger constants.
\end{abstract}

\section{Introduction}
There are a growing number of examples showing that ordered structures can be useful for solving geometric and topological problems that appear to be hard to analyze by traditional combinatorial methods. The aim of the present note is to provide an example concerning intersection patterns of curves, where one can apply ordered graphs.

First, we agree on the terminology. An \emph{ordered graph} $G_{<}$ is a graph $G$ with a total ordering $<$ on its vertex set. If the ordering $<$ is clear from the context, we write $G$ instead of $G_{<}$. An ordered graph $H_{<'}$ is an {\em induced subgraph} of the ordered graph $G_{<}$, if there exists an embedding $\phi: V(H)\rightarrow V(G)$ such that for every $u,v\in V(H)$, if $u<'v$ then $\phi(u)<\phi(v)$, and $uv\in E(H)$ if and only if $\phi(u)\phi(v)\in E(G)$.

A {\em monotone path} $P_k$ of length $k$ is an ordered graph with $k$ vertices $v_1<v_2<\ldots<v_k$ in which  $v_iv_j$ is an edge if and only if $|j-i|=1$. A \emph{bi-clique} in an (ordered or unordered) graph $G$ consists of a pair of disjoint subsets of the vertices $(A,B)$ such that $|A|=|B|$ and for every $a\in A$ and $b\in B$, there is an edge between $a$ and $b$. The size of a bi-clique $(A,B)$ is $|A|$.  A \emph{comparability graph} is a graph $G$ for which there exists a partial ordering on $V(G)$ such that two vertices are joined by an edge of $G$ if and only if they are comparable by this partial ordering. An \emph{incomparability graph} is the complement of a comparability graph. The maximum degree of the vertices of $G$ is denoted by $\Delta(G)$.

Our first theorem states that if a $P_k$-free ordered graph is not too dense, then its complement contains a large bi-clique.

\begin{theorem}\label{thm:path}
For every integer $k\ge 2$, there exists a constant $c=c(k)>0$ such that the following statement is true. Let $G_<$ be an ordered graph on $n$ vertices which satisfies $\Delta(G_<)<cn$ and does not have any induced ordered subgraph isomorphic to the monotone path $P_{k}$ of length $k$.

Then the complement of $G_<$ contains a bi-clique of size at least $cn/\log n$.
\end{theorem}

For the conclusion to hold, we need some upper bound on the degrees of the vertices (or on the number of edges) of the graph. To see this, consider the graph $G$ on the naturally ordered vertex set $\{1,\ldots,n\}$, in which $A=\{1,\ldots,\lfloor n/2\rfloor\}$ and $B=\{\lfloor n/2\rfloor +1,\ldots, n\}$ induce complete subgraphs, and any pair of vertices $a\in A, b\in B$ are joined by an edge randomly, independently with a very small probability $p>0$. This ordered graph has no induced monotone path of length $5$, its maximum degree satisfies $\Delta(G)<(1/2+p)n$, but the maximum size of a bi-clique in its complement is $O_p(\log n)$. Consequently, for the constant appearing in Theorem~\ref{thm:path}, we have $c_5\le 1/2$.

The assumption that $G_<$ contains no induced $P_3$ is equivalent to the property that $G_<$ is a comparability graph. In this special case (that is, for $k=3$), Theorem~\ref{thm:path} was established by Fox, Pach, and T\'oth \cite{FPT10}, and in a weaker form by Fox \cite{F06}. Apart from the value of the constant $c$, the bound is best possible for $k=3$ and, hence, for every $k\ge 3$.

An {\em ordered matching} is an ordered graph on $2k$ vertices which consists of $k$ edges, no two of which share an endpoint. Our next result is an analogue of Theorem~\ref{thm:path} for ordered graphs that contain no induced subgraph isomorphic to a fixed ordered matching.

\begin{theorem}\label{thm:matching}
For every ordered matching $M$, there exists a constant $c=c(M)>0$ such the following statement is true. Let $G_<$ be an ordered graph on $n$ vertices which satisfies $\Delta(G_<)<cn$ and does not have any induced ordered subgraph isomorphic to $M$.

Then the complement of $G_<$ contains a bi-clique of size at least $cn$.
\end{theorem}

The conclusion of Theorem~\ref{thm:matching} is stronger than that of Theorem~\ref{thm:path}: in this case we can find a linear-sized bi-clique in the complement of $G_<$.

Given a family of sets, $\mathcal{C}$, the \emph{intersection graph} of $\mathcal{C}$ is the graph, whose vertices correspond to the elements of $\mathcal{C}$, and two vertices are joined by an edge if and only if the corresponding sets have a nonempty intersection. A \emph{curve} is the image of a continuous function $\phi:[0,1]\rightarrow \mathbb{R}^{2}$. A curve is said to be \emph{$x$-monotone} if every vertical line intersects it in at most one point. Note that any convex set can be approximated arbitrarily closely by $x$-monotone curves, so the notion of $x$-monotone curve extends the notion of convex sets. Throughout this paper, a curve will be called a \emph{grounded} if one of its endpoints lies on the $y$-axis (on the vertical line $\{x=0\}$) and the whole curve is contained in the nonnegative half-plane $\{x\geq 0\}$. (By slight abuse of notation, we write $\{x\geq 0\}$ for the set $\{(x,y)\in\mathbb{R}^{2}: x\geq 0\}$.)

We will apply Theorems~\ref{thm:path} and~\ref{thm:matching} to give a simple combinatorial proof for the following Ramsey-type result of Fox and Pach~\cite{FP12}, which is related to a celebrated conjecture of Erd\H os and Hajnal~\cite{EH89, Ch}.

\begin{theorem} \cite{FP12}\label{thm:curves}
There exists an absolute constant $c>0$ with the following property. The intersection graph $G$ of any collection of $n$ $x$-monotone curves contains a bi-clique of size at least $cn/\log n$, or its complement $\overline{G}$ contains a bi-clique of size at least $cn$.
\end{theorem}

This result is tight, up to the value of $c$; see \cite{PT06}. Indeed, Fox \cite{F06} proved that for any $\varepsilon>0$ there exists a constant $c(\varepsilon)$ such that for every $n\in \mathbb{N}$, there exists an incomparability graph $G$ on $n$ vertices such that $G$ does not contain a bi-clique of size $c(\varepsilon)n/\log n$, and the complement of $G$ does not contain a bi-clique of size $n^{\epsilon}$. On the other hand, every incomparability graph is isomorphic to the intersection graph of a collection of $x$-monotone curves \cite{SiSiUr, Lo, PT06}.

It was shown in \cite{FPT10} that if the intersection graph of $n$ $x$-monotone curves has at most $2^{-18}\binom{n}{2}$ edges, then the second option holds in Theorem~\ref{thm:curves}: $\overline{G}$ contains a bi-clique of size at least $cn$. The proof of this statement uses a separator theorem for string graphs \cite{Lee, Ma}. The argument is rather involved and leaves no room for replacing $2^{-18}$ by a decent constant. Tomon \cite{T16} applied some properties of partially ordered sets to establish the upper bound $\left(\frac{1}{16}-o(1)\right)\binom{n}{2}$. Somewhat surprisingly, using ordered graphs, one can precisely determine the best constant for which the statement still holds.

\begin{theorem}\label{thm:threshold}
	For any $\epsilon>0$, there are constants $c_{1}=c_1(\epsilon),c_{2}=c_2(\epsilon)>0$, and an integer $n_0=n_0(\epsilon)$ such that the following statements are true. For every $n\ge n_0$,
	
	(1) there exist $n$ $x$-monotone curves such that their intersection graph $G$ has at most $(\frac{1}{4}+\epsilon)\binom{n}{2}$ edges, but the complement of $G$ does not contain a bi-clique of size $c_{1}\log n$;
	
	(2) for any $n$ $x$-monotone curves such that their intersection graph $G$ has at most $(\frac{1}{4}-\epsilon)\binom{n}{2}$ edges, the complement of $G$ contains a bi-clique of size $c_{2}n$.
\end{theorem}


It is easy to see that every intersection graph of {\em convex sets} in the plane is also an intersection graph of $x$-monotone curves. We prove (1) by constructing $n$ convex sets in the plane whose intersection graphs meets the requirements. Therefore, $\frac{1}{4}\binom{n}{2}$ is also a threshold for the emergence of linear sized bi-cliques in the complements of intersection graphs of convex sets.

In~\cite{FP12}, Theorem~\ref{thm:curves} was established in a more general setting: without assuming that the curves are $x$-monotone. It is a serious challenge to extend our proof to that case. We still believe that Theorem~\ref{thm:threshold} should also generalize to arbitrary curves.

\begin{conjecture}
	For any $\epsilon>0$, there exist $c_0=c_0(\epsilon)>0$ and $n_0=n_0(\epsilon)$ with the property that for any collection of $n\ge n_0$ curves whose intersection graph has at most $(\frac{1}{4}-\epsilon)\binom{n}{2}$ edges, the complement of $G$ contains a bi-clique of size $c_{0}n$.
\end{conjecture}

For {\em unordered} graphs without (unordered) induced paths of length $k$, the size of the largest bi-clique that can be found in $\overline{G}$ is larger than what was shown in Theorem~\ref{thm:path}: it is linear in $n$. More precisely, Bousquet, Lagoutte, and Thomass\'e~\cite{BLT15} proved that for every positive integer $k$, there exists $c(k)>0$ such that, if $G$ is an unordered graph with $n$ vertices and at most $c(k){n \choose 2}$ edges, which does not have an induced path of length $k$, then its complement $\overline{G}$ contains a bi-clique of size at least $c(k)n$. Recently, Chudnovsky, Scott, Seymour, and Spirkl~\cite{CSSS} generalized this result to any forbidden forest, instead of a path. In an upcoming work \cite{PT+}, we obtain similar extensions of Theorems \ref{thm:path} and \ref{thm:matching} to other {\em ordered forests}.

Our paper is organized as follows. Theorems \ref{thm:path} and~\ref{thm:matching} are proved in Sections \ref{sect:path} and \ref{sect:matching}, respectively. In Section~\ref{sect:curves}, we first establish Theorem \ref{thm:curves} for {\em grounded} $x$-monotone curves, and then show that this already implies the general result. Finally, we prove Theorem \ref{thm:threshold} in Section \ref{sect:threshold}.

\section{Ordered graphs avoiding a monotone induced path\\--Proof of Theorem \ref{thm:path}}\label{sect:path}

For any subset $U$ of the vertex set of a graph $G$, define the {\em neighborhood} of $U$, as
$$N(U)=\{v\in V(G)\setminus U:\exists u\in U\mbox{ such that }uv\in E(G)\}.$$
If $U$ consists of a single point $u$, we write $N(u)$ instead of $N(\{u\})$. The subgraph of $G$ {\em induced} by the vertices in $U$ is denoted by $G[U]$.

Given an ordered graph $G=G_{<}$ and two subsets $S,T\subset V(G)$, we write $S<T$ if $s<t$ for every $s\in S$ and $t\in T$. We say that a vertex $t\in T$ can be {\em reached} from a vertex $v\in V(G)$ {\em by a monotone $T$-path,} if there is an increasing sequence of vertices $v<t_1<t_2<\ldots<t_r=t$ such that $t_{1},\dots,t_{r}\in T$ and $vt_1, t_1t_2,\ldots, t_{r-1}t_r\in E(G)$. (The vertex $v$ does not necessarily belong to $T$.) Let $P_{G}(S,T)$ denote the set of vertices in $T$ that can be reached from {\em some} vertex in $S$ by a monotone $T$-path in $G$. If it is clear from the context what the underlying ordered graph $G$ is, we write $P(S,T)$ instead of $P_{G}(S,T)$. If $S$ consists of a single vertex $s$, we write $P(s,T)$ instead of $P(\{s\},T)$. Finally, if $T=V(G)$, then we write $P(S)$ instead of $P(S,V(G))$.

For the proof of Theorem \ref{thm:path}, we need the following lemma.

\begin{lemma}\label{lemma:path}
	Let $G=G_{<}$ be an ordered graph on the vertex set $S\cup T$, where $S<T$,  $|S|\geq \frac{n}{6\log_{2}n}$, and $|T|\geq n$. Then either there exists a vertex $v\in S$ such that $|P(v,T)|\geq \frac{n}{12}$ or the complement of  $G$ contains a bi-clique of size $\frac{n}{12\log_{2} n}$.
\end{lemma}

\begin{proof}
With no danger of confusion, we omit the use of floors and ceilings, whenever they are not crucial.
	Let $m=2^{k}$ such that $\frac{n}{12\log_{2} n}<m\leq \frac{n}{6\log_{2} n}$, and suppose that $\overline{G}$, the complement of $G$, contains no bi-clique of size $m$. Divide $T$ into $s=\frac{n}{3m}\geq2\log_{2} n$ intervals of size $3m$, denoted by $A_{1},\dots,A_{s}$, in this order.
	
	 We will recursively define a sequence of sets $S\supset S_{0}\supset S_{1}\supset\dots \supset S_{k}$ such that $|S_{i}|=2^{k-i}$ for $i=0,\dots,k$, and $|P(S_{i},T)\cap A_{i}|\geq m$ for $i=1,\dots,k$.

	 Let $S_{0}$ be an arbitrary $m$ element subset of $S$. Suppose that the set $S_{i}$ satisfying the above conditions has already been determined for some $i<k$. We define $S_{i+1}$, as follows. Let $X=P(S_{i},T)\cap A_{i}$. We can assume that $|N(X)\cap A_{i+1}|\geq 2m$, otherwise $|A_{i+1}\setminus N(X)|\geq m$ and there is a bi-clique $(A,B)$ in $\overline{G}$ of size $m$ such that $A\subset X$ and $B\subset A_{i+1}\setminus N(X)$. As $A_{i+1}$ is $<$-larger than every element of $A_{i}$, we have $|N(X)\cap A_{i+1}|\subset P(S_{i},T)\cap A_{i+1}$ and, hence, $|P(S_{i},T)\cap A_{i+1}|\geq 2m$.
	
	 Partition $S_{i}$ arbitrarily into two sets of size $|S_{i}|/2$, denoted by $S'$ and $S''$. Clearly, we have $P(S',T)\cup P(S'',T)=P(S_{i},T)$, so either $|P(S',T)\cap A_{i+1}|\geq m$, or $|P(S'',T)\cap A_{i+1}|\geq m$. In the first case, set $S_{i+1}=S'$, in the second, set $S_{i+1}=S''$.
	
	 At the end of the process, $S_{k}$ consists of one vertex, say $v$, and $|P(v,T)\cap A_{k}|\geq m$. We can actually assume that $|P(v,T)\cap A_{j}|\geq m$ for $j=k,\dots,s$. Indeed, there is no edge between $P(v,T)\cap A_{k}$ and $A_{j}\setminus P(v,T)$, so if $|A_{j}\cap P(v,T)|\leq m$, then $|A_{j}\setminus P(v,T)|\geq 2m$, which means that there exists a bi-clique $(A,B)$ of size $m$ in $\overline{G}$ such that $A\subset P(v,T)\cap A_{k}$ and $B\subset A_{j}\setminus P(v,T)$.
	
	 Summing up, we obtain that $$|P(v,T)|\geq \sum_{j=k}^{s}|P(v,T)\cap A_{j}|\geq m(s-k)>\frac{n}{12}.$$
\end{proof}

\begin{proof}[Proof of Theorem \ref{thm:path}]
Let $c=1/(24k^{2})$. Let $G=G_{<}$ be an ordered graph on $n$ vertices of maximum degree at most $cn$, and suppose that $\overline{G}$ does not contain a bi-clique of size $m=cn/\log_{2}n$. We have to prove that $G_{<}$ contains $P_{k}$ as an induced subgraph.

Let the vertex set of $G$ be $\{1,\dots,n\}$, and for $i=1,\dots,k$, let $A_{i}=\{(i-1)n/k+1,\dots,in/k\}$. We recursively construct a sequence of vertices $x_{1}<\dots<x_{k}$  that satisfy conditions (1) and (2) below, for $l=1,\dots,k$. Let
  $$U_{l+1}=V(G)\setminus\left(\bigcup_{i=1}^{l-1} N(x_{i})\right).$$
   Then

\noindent  (1)\;\;\;\; $\{x_{1},\dots,x_{l}\}$ is an induced copy of $P_{l}$,

\noindent  (2)\;\;\;\; $|P(x_{l},U_{l+1})\cap A_{l+1}|\geq m.$

For $l=1$, apply Lemma \ref{lemma:path} to the subgraph of $G$ induced by $A_{1}\cup A_{2}$ with $S=A_{1}$, $T=A_{2}$, and $n/k$ instead of $n$. Then there exists $x_{1}\in A_{1}$ such that $|P_{G}(x_{1})\cap A_{2}|\geq \frac{n}{12k}>m$.

Now let $l>1$ and suppose that the vertices $x_{1}<\dots<x_{l-1}$ satisfying conditions (1) and (2) have already been defined. Let $S=P(x_{l-1},U_{l})\cap A_{l}$ and $T=U_{l+1}\cap A_{l+1}$. (Note that for the definition of $U_{l+1}$ we do not need $x_{l}$.) Then $|S|\geq m$ and, as the maximum degree of $G$ is at most $cn$, we have $|T|\geq |A_{l+1}|-(l-1)cn>\frac{n}{2k}$. Apply Lemma \ref{lemma:path} to the subgraph of $G$ induced by $S\cup T$ with $n/2k$ instead of $n$. Since $\overline{G}$ does not contain a bi-clique of size $$\frac{cn}{\log_{2}n}<\frac{n/(2k)}{12\log_{2}(n/(2k))},$$ there exists $w\in S$ such that $|P(w,T)|>\frac{n}{24k}$. We have $w\in P(x_{l-1},U_{l})$, therefore $w$ can be reached from $x_{l-1}$ by a monotone $U_{l}$-path. Let $x_{l-1}=u_{0}<\dots<u_{r}=w$ be such a path with the minimum number of vertices. By the definition of $U_{l}$,  the vertices $u_{1},\dots,u_{r}\in U_{l}$ do not belong to the neighborhoods of $x_{1},\dots,x_{l-2}$, and, by the minimality of the path, $u_{2},\dots,u_{r}$ are not in the neighborhood of $x_{l-1}$. Setting $x_{l}=u_{1}$, we find that $\{x_{1},\dots,x_{l}\}$ is an induced copy of $P_{l}$. Thus, condition (1) is satisfied.

Vertex $w$ can be reached from $x_{l}$ by a monotone $U_{l+1}$-path, and every $z\in P(w,T)$ can be reached from $w$ by a monotone $U_{l+1}$-path. Therefore, every $z\in P(w,T)$ can be reached from $x_{l}$ by a monotone $U_{l+1}$-path. This yields that $$|P(x_{l},U_{l+1})\cap A_{l+1}|\geq |P(w,T)|>\frac{n}{24k}>m,$$
so that condition (2) is satisfied.

For $l=k$, the ordered subgraph of $G$ induced by $\{x_{1},\dots,x_{k}\}$ is isomorphic to $P_{k}$. This completes the proof of the theorem.
\end{proof}

\section{Ordered graphs avoiding an induced matching\\--Proof of Theorem \ref{thm:matching}}\label{sect:matching}


\begin{proof}[Proof of Theorem \ref{thm:matching}]
	Let $k$ be the number of edges of $M$, and set $c=1/(8k^{3})$. Let $G=G_{<}$ be an ordered graph on $n$ vertices such that the maximum degree of $G$ is at most $cn$, and suppose that $\overline{G}$ does not contain a bi-clique of size $cn$. We have to prove that $G$ contains $M$ as an induced subgraph.
	
	Suppose that $\{1,\dots,2k\}$ is the vertex set of $M$, and let $\{a_{1},b_{1}\},\dots\{a_{k},b_{k}\}$ be the edges of $M$. Let the vertex set of $G$ be $\{1,\dots,n\}$, and let $A_{1},\dots,A_{2k}$ be a partition of $V(G)$ into $2k$ intervals of size $\frac{n}{2k}$. Observe that, for $i=1,\dots,k$, there exists a set $E_i$ of $\frac{n}{4k}$ disjoint edges between $A_{a_{i}}$ and $A_{b_{i}}$. Otherwise, we could find a bi-clique $(A,B)$ of size $\frac{n}{4k}$ in $\overline{G}$ such that $A\subset A_{a_{i}}$ and $B\subset A_{b_{i}}$, contradicting our assumptions. Let $B_{a_{i}}\subset A_{a_{i}}$ and $B_{b_{i}}\subset A_{b_{i}}$ be the set of vertices incident to the edges of $E_{i}$.
	
	Pick an edge $e_{i}=\{u_{a_{i}},u_{b_{i}}\}$ randomly and uniformly from $E_{i}$ for $i=1,\dots,k$, and let $U=\{u_{1},\dots,u_{2k}\}$. To complete the proof, it is sufficient to show that with positive probability the subgraph of $G$ induced by $U$ is isomorphic to $M$.
	
	 Clearly, the subgraph induced by $U$ is not isomorphic to $M$ if and only if $\{u_{i},u_{j}\}$ is an edge of $G$ for some $\{i,j\}\not\in E(M)$. Let $1\leq i<j\leq 2k$ such that $\{i,j\}$ is not an edge of $M$. As the maximum degree of the vertices of $G$ is at most $cn$, there are at most $cn|B_{i}|$ edges between $B_{i}$ and $B_{j}$. As $u_{i}$ and $u_{j}$ are uniformly distributed in $B_{i}$ and $B_{j},$ and $u_{i}$ is independent of $u_{j}$, the probability that $\{u_{i},u_{j}\}$ is an edge of $G$ is at most $\frac{cn|B_{i}|}{|B_{i}||B_{j}|}<\frac{1}{2k^{2}}.$ Therefore, we have
	$$\mathbb{P}(\{u_{i},u_{j}\}\in E(G)\mbox{ for some }\{i,j\}\not\in E(M))\leq \sum_{\{i,j\}\not\in E(M)}\mathbb{P}(\{u_{i},u_{j}\}\in E(G))<\frac{\binom{2k}{2}}{2k^{2}}<1.$$
	
	Hence, with positive probability the subgraph of $G$ induced by $U$ is isomorphic to $M$, which implies that $G$ contains $M$ as an induced subgraph.		
\end{proof}

\section{Intersection graphs of curves--Proof of Theorem~\ref{thm:curves}}\label{sect:curves}

First, we prove Theorem~\ref{thm:curves} in the special case where the curves are grounded, that is, their left endpoints lie on the $y$-axis.

\begin{lemma}\label{lemma:grounded}
There exists an absolute constant $c>0$ with the following property. The intersection graph $G$ of any collection $\mathcal C$ of $n$ grounded $x$-monotone curves contains a bi-clique of size at least $cn/\log n$, or its complement $\overline{G}$ contains a bi-clique of size at least $cn$.
\end{lemma}

To prove this lemma, we first show that the intersection graphs of any collection of grounded $x$-monotone curves can be ordered in such a way that it has no ordered matching consisting of two intertwined edges, or its complement has no monotone path $P_{4}$.

Let $M_{1}$ denote the ordered matching on vertex set $\{1,2,3,4\}$, with edges $\{1,3\}$ and $\{2,4\}$.

\begin{lemma}\label{lemma:char}
	Let $\mathcal{C}$ be a family of grounded curves (not necessarily $x$-monotone), let $G$ be the intersection graph of  $\mathcal{C}$, and let $<$ be the total ordering of $\mathcal{C}$ according to the $y$-coordinates of the endpoints of the elements of $\mathcal{C}$ lying on $\{x=0\}$.
	
	(1) Then $G_{<}$ does not contain $M_{1}$ as an induced subgraph.
	
	(2) If, in addition, the elements of $\mathcal{C}$ are $x$-monotone curves, then $\overline{G}_{<}$ does not contain $P_{4}$ as an induced subgraph.
\end{lemma}

\begin{proof}
	(1) Suppose that $G_{<}$ contains $M_{1}$ as an induced subgraph, and let $\alpha_{1}<\alpha_{2}<\alpha_{3}<\alpha_{4}$ denote the curves corresponding to the vertices of $M_1$. As $\alpha_{1}$ and $\alpha_{3}$ intersect, the line $\{x=0\}$ and the two curves $\alpha_{1}$, $\alpha_{3}$ enclose a closed bounded region $A$. Curve $\alpha_{2}$ is disjoint from both $\alpha_{1}$ and $\alpha_{3}$, and its endpoint on $\{x=0\}$ belongs to $A$, we have $\alpha_{2}\subset A$. Curve $\alpha_{4}$ is also disjoint from $\alpha_{1}$ and $\alpha_{3}$, but its endpoint on $\{x=0\}$ is not in $A$, so $\alpha_{4}\cap A=\emptyset$. Hence, $\alpha_{2}$ and $\alpha_{4}$ cannot intersect, contradiction.
	
	(2) Suppose that $\overline{G}_{<}$ contains $P_{4}$ as an induced subgraph, and let $\alpha_{1}<\alpha_{2}<\alpha_{3}<\alpha_{4}$ denote the corresponding vertices. As before, the line $\{x=0\}$ and the two curves $\alpha_{1}$, $\alpha_{3}$ enclose a closed bounded region $A$, and $\alpha_{2}\subset A$. Since $\alpha_{1}$ and $\alpha_{3}$ are $x$-monotone, every vertical line intersecting $A$ intersects $\alpha_{1}$ and $\alpha_{3}$ in exactly one point, the intersection point with $\alpha_{1}$ lying below the intersection point with $\alpha_{3}$. Curve $\alpha_{4}$ is disjoint from $\alpha_{3}$, so for every vertical line intersecting $\alpha_{3}$ and $\alpha_{4}$, its intersection with $\alpha_{3}$ is below its intersection with $\alpha_{4}$. Therefore, we have $A\cap \alpha_{4}=\emptyset$, which implies that $\alpha_{2}$ and $\alpha_{4}$ are disjoint, contradiction. See Figure \ref{figure1} for an illustration.
\end{proof}

\begin{figure}[t]
	\begin{center}
\begin{tikzpicture}

\node (v1) at (-0.5,4.5) {};
\node (v2) at (-0.5,-2.5) {};
\draw  (v1) edge (v2);
\draw  plot[smooth, tension=.7] coordinates {(-0.5,2.5) (1,0.5) (2,1.5) (3,0.5) (3.5,3)};
\draw  plot[smooth, tension=.7] coordinates {(-0.5,-1) (1,-1.5) (1.5,-0.5) (2,2) (5,2.5)};
\draw  plot[smooth, tension=.7] coordinates {(-0.5,0) (0,0.5) (0.5,-0.5) (1,0)};
\draw  plot[smooth, tension=.7] coordinates {(-0.5,3.5) (1,3) (2,3) (3.5,4) (4.5,1.5)};
\node at (-0.75,2.5) {$\alpha_{3}$};
\node at (-0.75,-1) {$\alpha_{1}$};
\node at (-0.75,0) {$\alpha_{2}$};
\node at (-0.75,3.5) {$\alpha_{4}$};
\end{tikzpicture}
\caption{An illustration for the proof of part (2) of Lemma \ref{lemma:char}.}
\label{figure1}
\end{center}
\end{figure}
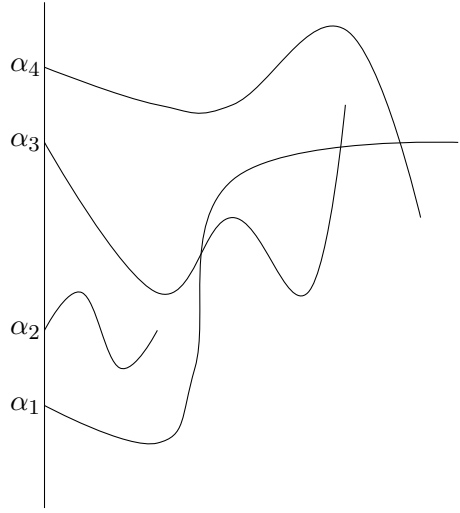

In view of Lemma~\ref{lemma:char}, we may be able to use Theorem~\ref{thm:path} or Theorem~\ref{thm:matching} to argue that $G$, the intersection graph of a collection of $n$ $x$-monotone curves, contains a bi-clique of size $\Omega(n/\log n)$, or its complement, $\overline{G}$, contains a bi-clique of size $\Omega(n)$. However, to apply one of these two theorems, either in $G$ or in $\overline{G}$, the maximum degree of the vertices must be sufficiently small, which is not necessarily the case.

To overcome this difficulty, we use the following statement which guarantees that $G$ or $\overline{G}$ has a large induced subgraph with very few edges.

\begin{lemma}\label{lemma:ordereduni}
	Let $H_{<}$ be an ordered graph and let $\epsilon>0$. Then there exists a constant $c_0=c_0(H_{<},\epsilon)>0$ such that every ordered graph $G_{<}$ on $n$ vertices that does not contain $H_{<}$ as an induced subgraph has the following property. There is a subset $U\subset V(G)$ with $|U|\geq c_0n$ such that either $|E(G[U])|\leq \epsilon\binom{|U|}{2}$ or $|E(G[U])|\geq (1-\epsilon)\binom{|U|}{2}$ holds.
\end{lemma}

Lemma~\ref{lemma:ordereduni} is an easy consequence of the unordered variant of the same statement due to R\"odl \cite{R86} and a result of R\"odl and Winkler~\cite{RW}.

\begin{lemma} \cite{R86}\label{lemma:universal}
	Let $H$ be a graph and let $\epsilon>0$. Then there exists a constant $c_1=c_1(H,\epsilon)>0$ such that every graph $G$ on $n$ vertices that does not contain $H$ as an induced subgraph has the following property. There is a subset $U\subset V(G)$ with $|U|\geq c_1n$ such that either $|E(G[U])|\leq \epsilon\binom{|U|}{2}$ or $|E(G[U])|\geq (1-\epsilon)\binom{|U|}{2}$ holds.
\end{lemma}

\begin{lemma} \cite{RW}\label{rodlwinkler}
For every ordered graph $H_<$, there exists an unordered graph $H'$ with the property that for any total ordering $\prec$ on $V(H')$, the ordered graph $H'_{\prec}$ contains $H_{<}$ as an induced subgraph.
\end{lemma}

\begin{proof}[Proof of Lemma~\ref{lemma:ordereduni}]
Let $H'$ be the graph whose existence is ensured by Lemma~\ref{rodlwinkler}. For every ordered graph $G_{<}$ that does not contain $H_{<}$ as an induced subgraph, the underlying unordered graph $G$ does not contain $H'$ as an induced subgraph. Hence, the statement is true with $c_0=c_1(H',\epsilon)$, where $c_1(H',\epsilon)$ is the constant defined in Lemma \ref{lemma:universal}.
\end{proof}

\begin{corollary}\label{cor:universal}
		Let $H_{<}$ be an ordered graph and let $\delta>0$. There exists a constant $c_2=c_2(H_{<},\delta)>0$ such that every  ordered graph $G_{<}$ on $n$ vertices that does not contain $H_{<}$ as an induced subgraph has the following property. There is a subset $U\subset V(G)$ with $|U|\geq c_2n$ such that either $\Delta(G[U])\leq \delta|U|$ or $\Delta(\overline{G}[U])|\leq \delta|U|$ holds.
\end{corollary}

\begin{proof}
Let $\epsilon=\delta/4$, and let $c_0=c_0(H_{<},\epsilon)$ be the constant given by Lemma \ref{lemma:ordereduni}. We show that $c_2=c_0/2$ meets the requirements.

Let $G_{<}$ be an ordered graph on $n$ vertices that does not contain an induced copy of $H_{<}$.  Then there exists $U_{0}\subset V(G)$ with $|U_{0}|\geq c_0n$ such that either $|E(G[U_{0}])|\leq \epsilon\binom{|U_{0}|}{2}$ or $|E(G[U_{0}])|\geq (1-\epsilon)\binom{|U_{0}|}{2}$ holds. Without loss of generality, suppose that $|E(G[U_{0}])|\leq \epsilon\binom{|U_{0}|}{2}$; the other case can be handled in a similar manner. Let $U_{1}$ be the set of vertices $u\in U_{0}$ whose degree in $G[U_{0}]$ is larger than $2\epsilon |U_{0}|$. Clearly, we have $|U_{1}|<|U_{0}|/2$. Setting $U=U_{0}\setminus U_{1}$, we obtain $|U|> |U_{0}|/2\geq c_2n$, and the degree of every vertex in $G[U]$ is at most $2\epsilon |U_{0}|<4\epsilon |U|=\delta|U|$.
\end{proof}

Now we are in a position to prove Lemma \ref{lemma:grounded}.

\begin{proof}[Proof of Lemma \ref{lemma:grounded}]
	Let $c'=c(M_{1})$ be the constant defined in Theorem \ref{thm:matching}, and let $c''=c(4)$ be the constant defined in Theorem \ref{lemma:path}. Set $\delta=\min\{c',c''\}$, and let $c_2=c_2(M_1,\delta)$ be the constant defined in Corollary \ref{cor:universal}.
		
	By Lemma \ref{lemma:char} (1), there exists an ordering $<$ on ${G}$ such that $G_{<}$ does not contain $M_{1}$ as an induced subgraph. Hence, there exists $U\subset V(G)$ such that $|U|\geq c_{2}n$, and either $\Delta(G[U])<\delta |U|$, or $\Delta(\overline{G}[U])<\delta|U|$. In the first case, $\overline{G}[U]$ contains a bi-clique of size $c'|U|\geq c'c_{2}n$. In the second case, by Lemma \ref{lemma:char} (2), $\overline{G}_{<}$ does not contain $P_{4}$ as an induced subgraph, so $G[U]$ contains a bi-clique of size $c''|U|/\log |U|\geq c''c_{2}n/\log n$. Thus, the statement is true with $c=\delta c_2$.
\end{proof}

Next, we show that Theorem \ref{thm:curves} holds not only for families of grounded curves, but also for families $\mathcal{C}$ of $x$-monotone curves, each of which intersects the same vertical line. Clearly, such a line splits $\mathcal{C}$ into two families of grounded curves, and the intersection graph of $\mathcal{C}$ is the union of the intersection graphs of these two families. In order to exploit this property, we make use of the following technical lemma. The constants $c$ and $c'$ appearing in Lemma~\ref{lemma:hereditary} are different from all previously used constants denoted by the same letters. A similar lemma was established in \cite{FP09}, but it is not suitable for our purposes.

A family of graphs $\mathcal{G}$ is called \emph{hereditary}, if for every $G\in\mathcal{G}$, every induced subgraph of $G$ is also a member of $\mathcal{G}$. 
For any pair of graphs $G_{1}$ and $G_{2}$ with  $V(G_{1})=V(G_{2})$, the \emph{union of $G_{1}$ and $G_{2}$} is defined as the graph $G_{1}\cup G_{2}$ whose vertex set is $V(G_{1})$ and edge set is $E(G_{1})\cup E(G_{2})$.

\begin{lemma}\label{lemma:hereditary}
	Let $\mathcal{G}$ be a hereditary family of graphs. Suppose that there exist a constant $c, 0<c<1,$ and a monotone increasing function $f:\mathbb{N}\rightarrow \mathbb{R}^{+} $ such that each member $G\in \mathcal{G}$ on $n$ vertices contains either a bi-clique of size at least $n/f(n)$, or $\overline{G}$ contains a bi-clique of size at least $cn$.

Then there exists a constant $c'>0$ with the following property. If $G_{1},G_{2}\in\mathcal{G}$, $V(G_{1})=V(G_{2})$, and $|V(G_{1})|=n$, then $G_{1}\cup G_{2}$ contains a bi-clique of size at least $c'n/f(n)$ or the complement of $G_{1}\cup G_{2}$ contains a bi-clique of size at least $c'n$.
\end{lemma}

\begin{proof}
Let $k=1+\lceil\log_{2}(1/c)\rceil$. We show that the constant $c'=c^{k+1}/2$ will meet the requirements. Let $G_{1},G_{2}\in\mathcal{G}$ such that $V=V(G_{1})=V(G_{2})$ and $|V|=n$.

We can suppose that if $U\subset V$ such that $|U|\geq \frac{c'}{c}n$, then both $\overline{G}_{1}[U]$ and $\overline{G}_{2}[U]$ contain a bi-clique of size $c|U|$. Indeed, otherwise, either $G_{1}[U]$ or $G_{2}[U]$ contains a bi-clique of size $c|U|/f(|U|)\geq c'n/f(n)$, so $G_{1}\cup G_{2}$ also contains a bi-clique of size $c'n/f(n)$, and we are done.	

For $i=0,\dots,k$, we define disjoint sets $U_{i,1},\dots,U_{i,2^{i}}\subset V$ such that $|U_{i,j}|\geq c^{i}n$ for $j=1,\dots,2^{i}$, and there is no edge between $U_{i,j}$ and $U_{i,j'}$ in $G_{1}$ for $1\leq j<j'\leq 2^{i}$. Let $U_{0,1}=V$. If $U_{i,1},\dots,U_{i,2^{i}}$ are already defined for $i<k$, let $(U_{i+1,2j-1},U_{i+1,2j})$ be a bi-clique of size $c|U_{i,j}|$ in $\overline{G}_{1}[U_{i,j}]$. As $|U_{i,j}|=c^{i}n>\frac{c'}{c}n$, such a bi-clique always exists.

Now let $U=\bigcup_{j=1}^{2^{k}}U_{k,j}$. Then $|U|=2^{k}c^{k}n> \frac{c'}{c}n$, so $\overline{G}_{2}[U]$ contains a bi-clique $(A,B)$ of size at least $c|U|$. Therefore, there exists $1\leq j\leq 2^{k}$ such that $|U_{k,j}\cap A|\geq |A|/2^{k}\geq c|U|/2^{k}=c^{k+1}n>c'n$, and there exists $1\leq j'\leq k$ such that $j\neq j'$ and
  $$|U_{k,j'}\cap B|\geq \frac{|B|-|U_{k,j}|}{2^{k}}\geq \frac{c|U|-|U|/2^{k}}{2^{k}}= \left(c^{k+1}-\frac{c^{k}}{2^{k}}\right)n\geq \frac{c^{k+1}n}{2}=c'n.$$
There is no edge between $A\cap U_{k,j}$ and $B\cap U_{k,j'}$ in $\overline{G}_{1}$ and $\overline{G}_{2}$, so the complement of $G_{1}\cup G_{2}$ contains a bi-clique of size $c'n$.
\end{proof}

Now we can prove Theorem \ref{thm:curves} for collections of $x$-monotone curves that intersect the same vertical line.

\begin{lemma}\label{lemma:line}
   Let $\mathcal{C}$ be a collection of $n$ $x$-monotone curves such that each member of $\mathcal{C}$ intersects a vertical line $l$. Let $G$ be the intersection graph of $\mathcal{C}$. Then either $G$ contains a bi-clique of size $\Omega(n/\log n)$, or the complement of $G$ contains a bi-clique of size $\Omega(n)$.
\end{lemma}

\begin{proof}
	Let $\mathcal{G}$ be the family of intersection graphs of collections of grounded $x$-monotone curves. Clearly, $\mathcal{G}$ is hereditary. By Lemma \ref{lemma:grounded}, there exists a constant $c>0$ such that each $G_{0}\in\mathcal{G}$ on $n$ vertices contains either a bi-clique of size $cn/\log n$, or the complement of $G_{0}$ contains a bi-clique of size $cn$. Thus, by Lemma \ref{lemma:hereditary}, there exists a constant $c'>0$ such that if $G_{1},G_{2}\in\mathcal{G}$ with $V(G_{1})=V(G_{2})$ and $|V(G_{1})|=n$, then either $G_{1}\cup G_{2}$ contains a bi-clique of size $c'n/\log n$, or the complement of $G_{1}\cup G_{2}$ contains a bi-clique of size $c'n$.
	
	The vertical line $l$ cuts each $x$-monotone curve $\alpha\in\mathcal{C}$ into a left and a right part, denoted by $\alpha_{1}$ and $\alpha_{2}$. Let $\mathcal{C}_{1}=\{\alpha_{1}:\alpha\in\mathcal{C}\}$, $\mathcal{C}_{2}=\{\alpha_{2}:\alpha\in\mathcal{C}\}$, and let $G_{1}$ and $G_{2}$ be the intersection graphs of $\mathcal{C}_{1}$ and $\mathcal{C}_{2}$, respectively. Then $G_{1}, G_{2}\in\mathcal{G}$ and $G=G_{1}\cup G_{2}$, so we are done.
\end{proof}

Finally, everything is ready to prove our main theorem.

\begin{proof}[Proof of Theorem \ref{thm:curves}]
	For each $\alpha\in\mathcal{C}$, let  $r(\alpha)$ denote the $x$-coordinate of the right endpoint of $\alpha$. Without loss of generality, we can suppose that $r(\alpha)\neq r(\alpha')$  for $\alpha\neq \alpha'$. Let $\alpha_{1},\dots,\alpha_{n}$ be the enumeration of the curves in $\mathcal{C}$ such that $r(\alpha_{1})<\dots<r(\alpha_{n})$.
	
	Set $m=\lfloor n/3\rfloor$ and consider a vertical line $l=\{x=r\}$, where $r(\alpha_{m})<r<r(\alpha_{m+1})$. Let $\mathcal{C}'$ denote the set of curves in $\mathcal{C}$ which have a nonempty intersection with $l$. We distinguish two cases.
	
	{\em Case 1:} $|\mathcal{C}'|\geq m$. Let $G'$ be the intersection graph of $\mathcal{C}'$. Then by Lemma \ref{lemma:line}, either $G'$ contains a bi-clique of size $\Omega(m/\log m)=\Omega(n/\log n)$, or $\overline{G}'$ contains a bi-clique of size $\Omega(m)=\Omega(n)$.
	
	{\em Case 2:} $|\mathcal{C}'|<m$. Let $A=\{C_{i}:i\leq m\}$ and $B=\mathcal{C}\setminus(A\cup \mathcal{C}')$. Then $|B|\geq n/3$, and no curve in $A$ intersects any curve in $B$, because $A$ and $B$ are separated by $l$. Hence, $\overline{G}$ contains a bi-clique of size $m=\Omega(n)$.
\end{proof}

\section{Sharp threshold for intersection graphs--Proof of Theorem \ref{thm:threshold}}\label{sect:threshold}

In this section, we prove Theorem \ref{thm:threshold}. Part (1) of the theorem is an easy consequence of the following result of Pach and T\'oth \cite{PT06+}; see also \cite{PRY18}.

\begin{lemma}(Pach, T\'oth \cite{PT06+})\label{lemma:4partition}
	Let $V$ be an $n$-element set and let $V_{1},V_{2},V_{3},V_{4}$ be a partition of $V$ into $4$ sets. Let $G$ be a graph on the vertex set $V$ such that $V_{i}$ spans a clique in $G$ for $i=1,2,3,4$.

Then $G$ can be realized as the intersection graph of convex sets.
\end{lemma}

\begin{proof}[Proof of Theorem \ref{thm:threshold}, part (1)]
	Let $V$ be an $n$-element set and let $V_{1},V_{2},V_{3},V_{4}$ be a partition of $V$ into four sets of size roughly $n/4$. Consider the graph $G$ in which $V_{1},V_{2},V_{3},V_{4}$ are cliques, and any pair of vertices $\{u,v\}$, where $u\in V_{i}$ and $v\in V_{j}$ with $i\neq j$ is joined by an edge with probability $\epsilon$. Then with probability tending to $1$, $G$ has at most $(\frac{1}{4}+\epsilon)\binom{n}{2}$ edges, and $\overline{G}$ contains no bi-clique of size $4\frac{\log n}{\epsilon}$. By Lemma \ref{lemma:4partition}, $G$ can be realized as the intersection graph of convex sets and, therefore, by $x$-monotone curves.
\end{proof}

 In the rest of the section, we prove part (2) of Theorem \ref{thm:threshold}. For the proof, we use the following characterization of intersection graphs of $x$-monotone curves that intersect the same vertical line, which was established in \cite{PT18}.

 A graph $G_{<_{1},<_{2}}$  with two total orderings, $<_{1}$ and $<_{2}$, on its vertex set is called \emph{double-ordered}. If the orderings $<_{1},<_{2}$ are clear from the context, we shall write $G$ instead of $G_{<_{1},<_{2}}$.

 \begin{definition}
 	A double-ordered graph $G_{<_{1},<_{2}}$ is called \emph{magical} if for any three distinct vertices $a,b,c\in V(G)$ with $a<_{1}b<_{1}c$ the following is true: if $ab,bc\in E(G)$ and $ac\not\in E(G)$, then $b<_{2}a$ and $b<_{2}c$.
 	A graph $G$ is said to be {\em magical} if there exist two total orders $<_{1},<_{2}$ on $V(G)$ such that $G_{<_{1},<_{2}}$ is magical.
 \end{definition}

A \emph{triple-ordered} graph is a graph $G_{<_{1},<_{2},<_{3}}$ with three total orders $<_{1},<_{2},<_{3}$ on its vertex set.

\begin{definition}\label{2magic}
	A {\em triple-ordered graph} $G_{<_{1},<_{2},<_{3}}$ is called \emph{double-magical}, if there exist two magical graphs $G^{1}_{<_{1},<_{2}}$ and $G^{2}_{<_{1},<_{3}}$ on $V(G)$  such that  $E(G_{<_{1},<_{2},<_{3}})=E(G^{1}_{<_{1},<_{2}})\cap E(G^{2}_{<_{1},<_{3}})$. An {\em unordered graph} $G$ is said to be {\em double-magical} if there exist three total orders $<_{1},<_{2},<_{3}$ on $V(G)$ such that the triple-ordered graph $G_{<_{1},<_{2},<_{3}}$ is double-magical.
\end{definition}

\begin{lemma}(Pach, Tomon \cite{PT18})\label{dmagical2}
	A graph is double-magical if and only if it isomorphic to the complement of the intersection graph of a collection of $x$-monotone curves, each of which intersects a vertical line $l$.
\end{lemma}

Relying on this characterization, part (2) of Theorem \ref{thm:threshold} reduces to the following lemma about double-magical graphs.

\begin{lemma}\label{lemma:2magical}
	For any $\epsilon>0$, there exists a constant $c=c(\epsilon)>0$ with the following property. For every positive integer $n$, every double-magical graph with $n$ vertices and at least $(\frac{3}{4}+\epsilon)\binom{n}{2}$ edges contains a bi-clique of size $cn$.
\end{lemma}

\begin{proof}
	Let $G_{<_{1},<_{2}}$ and $G_{<_{1},<_{3}}$ be magical graphs such that $E(G)=E(G^{1}_{<_{1},<_{2}})\cap E(G^{2}_{<_{1},<_{3}})$.
	
	
	A triple of vertices $(a,b,c)$ in $G$ is called an {\em $i$-hole} for $i=2,3$, if $a<_{1}b<_{1}c$ and $b<_{i}a$ and $b<_{i}c$.
A $4$-tuple $(a,b,b',c)$ of vertices of $G$ is said to be \emph{forcing} if
	\begin{enumerate}
		\item $a<_{1}b<_{1}c$,
		\item $a<_{1}b'<_{1}c$,
		\item $(a,b,c)$ is not a $2$-hole,
		\item $(a,b',c)$ is not a $3$-hole.
	\end{enumerate}
	  Note that we do not exclude that $b=b'$. If $(a,b,b',c)$ is forcing, we say that the {\em set} $\{a,b,b',c\}$ is also {\em forcing}. We are interested in forcing $4$-tuples for the following reason: if $(a,b,b',c)$ is forcing such that $ab,bc,ab',b'c$ are edges of $G$, then $ac$ is also an edge. Indeed, if $ab,bc$ are edges of $G$, then $ab,bc$ are edges of $G_{<_{1},<_{2}}$. In this case, as $G_{<_{1},<_{2}}$ is magical and $(a,b,c)$ is not a $2$-hole, $ac$ is also an edge of $G_{<_{1},<_{2}}$. Similarly, $ab',b'c$ are edges of $G_{<_{1},<_{3}}$. Then, as $G_{<_{1},<_{3}}$ is magical and $(a,b',c)$ is not a 3-hole, $ac$ is also an edge of $G_{<_{1},<_{3}}$. Therefore, $ac$ is an edge of $G$ as well.
	
	  Let the {\em order type} of a $4$-tuple $(a,b,b',c)$ of vertices be $(s_{1},s_{2},s_{3},s_{4})$, where\\
	 $$s_{1}=\begin{cases}+ &\mbox{if } a<b\\ - &\mbox{if } b<a\end{cases},\;\;
s_{2}=\begin{cases}+ &\mbox{if } b<c\\ - &\mbox{if } c<b\end{cases},\;\;
s_{3}=\begin{cases}+ &\mbox{if } a<b'\\ - &\mbox{if } b'<a\end{cases},\;\; \mbox{and }\;\;
s_{4}=\begin{cases}+ &\mbox{if } b'<c\\ - &\mbox{if } c<b'\end{cases}.$$\\
	 Note that if $(a,b,b',c)$ is a $4$-tuple such that $a<_{1}b<_{1}c$ and $a<_{1}b'<_{1}c$, and the  order type of $(a,b,b',c)$ is $(s_{1},s_{2},s_{3},s_{4})$, then $(a,b,b',c)$ is not forcing if and only if $(s_{1},s_{2})=(-,+)$, or $(s_{3},s_{4})=(-,+)$.

	\begin{claim}
		Every set of $5$ vertices in $G$ contains a forcing $4$-tuple.
	\end{claim}
	\begin{proof}
		There are $(5!)^{2}=14400$ non-isomorphic triple orderings of a $5$ elements set, so it is sufficient to show that each of them contains a forcing $4$-tuple. A quick computer search shows that this is indeed the case.
	\end{proof}
    As usual, let $K_{t}$ denote the complete graph on $t$ vertices. By a well known result of Erd\H{o}s and Simonovits \cite{ES}, the condition $|E(G)|\geq (1-\frac{1}{4}+\epsilon)\binom{n}{2}$ implies that $G$ contains at least $c_{0}n^{5}$ copies of the complete graph $K_{5}$, where $c_{0}=c_{0}(\epsilon)$ depends only on $\epsilon$.
	
	Now each copy of $K_{5}$ in $G$ contains a forcing $4$-tuple, which spans either a copy of $K_{4}$ or $K_{3}$ in $G$ (depending on whether $b=b'$). There are $16$ order types of $4$-tuples. Hence, there is  an order type $\tau$ such that at least $c_{0}n^{4}/32$ copies of $K_{4}$ in $G$ are forcing with order type $\tau$,  or at least $c_{0}n^{3}/32$ copies of $K_{3}$ in $G$ are forcing with order type  $\tau$.	
	
	In the first case, we deduce that there exists a pair of vertices $(b,b')$ in $G$ such that $b\neq b'$, and there are at least $c_{0}n^{2}/32$ pairs of vertices $(a,c)$ such that $(a,b,b',c)$ is forcing with order type $\tau$, and $\{a,b,b',c\}$ spans a copy of $K_{4}$. Let $A$ be the set of vertices $a$ that appear in such a forcing $4$-tuple $(a,b,b',c)$, and let $C$ be the set of vertices $c$ that appear in such a forcing $4$-tuple $(a,b,b',c)$. Then $|A||C|\geq c_{0}n^{2}/32$, so $|A|,|C|\geq c_{0}n/32$. If $a_{0}\in A$, there exists $c\in C$ such that $\{a_{0},b,b',c\}$ spans a copy of $K_{4}$, so $a_{0}$ is joined to $b$ and $b'$ by an edge. Similarly, every $c_{0}\in C$ is also connected to $b$ and $b'$ by an edge. Finally, for every $a_{0}\in A$ and $c_{0}\in C$, the $4$-tuple $(a_{0},b,b',c_{0})$ has order type $\tau$. But whether a $4$-tuple is forcing depends only on its order type, so $(a_{0},b,b',c_{0})$ is forcing. Hence, $a_{0}c_{0}$ is an edge for every $a_{0}\in A$ and $c_{0}\in C$, so $A\cup C$ spans a bi-clique of size at least $c_{0}n/32$.
	
	In the second case, there exist a vertex $b$ and at least $c_{0}n^{2}/32$ pairs of vertices $(a,c)$ such that $(a,b,b,c)$ is forcing with order type $\tau$, and $\{a,b,c\}$ spans a copy of $K_{3}$. Now we can proceed in the same way as in the previous case to find a bi-clique of size $c_{0}n/32$.	
	
	Hence, Lemma~\ref{lemma:2magical} holds with $c=c_{0}/32$.
\end{proof}

\begin{corollary}\label{cor:line}
	For any $\epsilon>0$, there exist a constant $c=c(\epsilon)>0$ and an integer $n_0=n_0(\epsilon)$ such that the following statement is true. For any $n\ge n_0$ $x$-monotone curves that intersect the same vertical line, if the intersection graph $G$ of the curves has at most $(\frac{1}{4}-\epsilon)\binom{n}{2}$ edges, then the complement of $G$ contains a bi-clique of size $cn$.
\end{corollary}

\begin{proof}
	By Lemma \ref{dmagical2}, the complement of $G$ is a double-magical graph. Since $\overline{G}$ has at least $(\frac{3}{4}+\epsilon)\binom{n}{2}$ edges, by Lemma \ref{lemma:2magical} it must contain a bi-clique of size $cn$.
\end{proof}

Similarly as in the proof of Theorem \ref{thm:curves}, we complete the proof of Theorem \ref{thm:threshold} by reducing the general configuration of $x$-monotone curves to the case, where every $x$-monotone curve has nonempty intersection with the same vertical line $l$.

 \begin{proof}[Proof of Theorem \ref{thm:threshold}, part (2)]
 	Without loss of generality, assume that $\epsilon<1/2$. Let $\mathcal{C}$ denote our collection of curves. For each $\alpha\in\mathcal{C}$, let  $r(\alpha)$ be the $x$-coordinate of the right endpoint of $\alpha$. We can also suppose that $r(\alpha)\neq r(\alpha')$  for $\alpha\neq \alpha'$. Let $\alpha_{1},\dots,\alpha_{n}$ be the enumeration of the curves in $\mathcal{C}$ such that $r(\alpha_{1})<\dots<r(\alpha_{n})$.
 	
 	Set $m=\epsilon n/2$ and consider a vertical line $l=\{x=r\}$, where $r(\alpha_{m})<r<r(\alpha_{m+1})$. Let $\mathcal{C}'$ denote the set of curves in $\mathcal{C}$ which have a nonempty intersection with $l$. We distinguish two cases.
 	
 	{\em Case 1:} $|\mathcal{C}'|\geq (1-\epsilon)n$. Let $G'$ be the intersection graph of $\mathcal{C}'$. Then $G'$ has at most $\left(\frac{1}{4}-\frac{\epsilon}{4}\right)\binom{|V(G')|}{2}$ edges. Therefore, the complement of $G'$ contains a bi-clique of size $c(\frac{\epsilon}{4})|V(G')|>c(\frac{\epsilon}{4})n/2$, where $c$ is the constant defined in Corollary \ref{cor:line}.
 	
 	{\em Case 2:} $|\mathcal{C}'|<(1-\epsilon)n$. Let $A=\{C_{i}:i\leq m\}$ and $B=\mathcal{C}\setminus(A\cup \mathcal{C}')$. Then $|B|\geq \epsilon n/2$, and no curve in $A$ intersects any curve in $B$, because $A$ and $B$ are separated by $l$. Hence, $\overline{G}$ contains a bi-clique of size $ \epsilon n/2$.
 \end{proof}

\end{document}